\documentclass[a4paper, 11pt]{amsart}
\usepackage{multicol}
\usepackage{breqn}
\usepackage{wasysym}
\usepackage{bookmark}

\usepackage[dvipsnames]{xcolor}

\usepackage[UKenglish]{babel}\usepackage[UKenglish]{isodate}\cleanlookdateon

\usepackage{caption}
\numberwithin{equation}{section}
\usepackage[a4paper,vmargin={2cm,2cm},hmargin={2cm,2cm},bottom=20mm]{geometry}
\usepackage{amsmath,amsfonts,amscd,amssymb,amsthm,color}
\usepackage[foot]{amsaddr}
\usepackage[mathscr]{euscript}
\usepackage{cancel} 
\usepackage{verbatim}
\usepackage{hyperref}
\usepackage{todonotes}
\usepackage{enumitem}
\usepackage[tikz]{mdframed}
\usepackage{marginnote}
\usepackage{dsfont}

\usepackage{thmtools} \usepackage[capitalize]{cleveref}

\usepackage{enumitem}
\setlist[itemize]{itemsep=2pt, topsep=3pt, parsep=0pt, partopsep=0pt}

\usepackage{enumitem}
\setlist[enumerate]{itemsep=2pt, topsep=3pt, parsep=0pt, partopsep=0pt}

\usepackage{color}
\definecolor{Red}{cmyk}{0,1,1,0}

\usepackage{color}
\definecolor{Purple}{cmyk}{79,100,0,0}

\declaretheorem[numberwithin=section]{theorem}

\declaretheorem[sibling=theorem]{lemma}

\declaretheorem[sibling=theorem]{remark}

\crefname{theorem}{Theorem}{Theorems}
\crefname{example}{Example}{Examples}
\crefname{lemma}{Lemma}{Lemmas}
\crefname{proposition}{Proposition}{Propositions}
\crefname{definition}{Definition}{Definitions}
\crefname{remark}{Remark}{Remarks}
\crefname{conjecture}{Conjecture}{Conjectures}
\crefname{openquestion}{Open Question}{Open Questions}
\crefname{equation}{}{}
\Crefname{equation}{}{}

\newcommand{\ind}[1]{{\mathds{1}_{#1}}}

\newcommand{\bra}[1]{{\langle #1 \rangle}}
\newcommand{\Bbra}[1]{{\big\langle #1 \big\rangle}}

\newcommand{\ZZ}{{\mathbb{Z}}}
\newcommand{\RR}{{\mathbb{R}}}
\newcommand{\II}{\mathcal{I}}
\newcommand{\NN}{{\mathbb{N}}}
\newcommand{\PP}{{\mathbb{P}}}
\newcommand{\EE}{{\mathbb{E}}}

\newcommand{\bx}{{\mathbf{x}}}

\newcommand{\ba}{{\mathbf{a}}}
\newcommand{\bla}{\boldsymbol{\lambda}}
\newcommand{\oo}{\mathsf{o}}

\newcommand{\Z}{\mathbb{Z}}
\renewcommand{\P}{\mathbb{P}}

\title{{I}mry-{M}a phenomenon for the hard-core model on $\mathbb{Z}^{2}$}
\author{Irene Ayuso Ventura$^\ast$} 
\thanks{$^\ast$Department of Mathematics, Durham University, Durham, UK} 
\author{ 
Leandro Chiarini$^\dagger$} 
\thanks{$^\dagger$Instituto de Matemática e Estatística, Universidade de São Paulo, São Paulo, Brazil} 
\author{Tyler Helmuth$^\ast$}
\author{Ellen Powell$^\ast$}

\begin{document}

\maketitle

\vspace{-1cm}
\begin{abstract}
	The \emph{Imry-Ma phenomenon} refers to the dramatic effect that
	disorder can have on first-order phase transitions for
	two-dimensional spin systems. The most famous example is the absence
	of a phase transition for the two-dimensional random-field Ising
	model. This paper establishes that a similar phenomena takes place
	for the hard-core model, a discrete model of crystallization:
	arbitrarily weak disorder prevents the formation of a crystal. Our
	proof of this behaviour is an adaptation of the Aizenman-Wehr
	argument for the Imry-Ma phenomenon, with the use of internal (spin
	space) symmetries for spin systems being replaced by the use spatial symmetries.
\end{abstract}

\section{Introduction}

The \emph{hard-core model} is particle system in which the vertices
$V$ of a graph $G$ are either occupied or not. The \emph{hard-core
	constraint} is that the occupied vertices form an \emph{independent
	set}, i.e., no two neighbouring vertices may be simultaneously
occupied.  Given an activity $\lambda\geq 0$, one obtains a
probability law by declaring the probability of each independent set
$I\subset V$ to be proportional to $\lambda^{|I|}$.

In this paper, we focus on the case that $G$ is $\Z^{2}$.  Physically,
the hard-core model in this setting serves as a model for the
adsorption of atoms onto a crystal surface, see,
e.g.,~\cite{Taylor1985}.  The adsorbed atoms are represented by
occupied vertices. The independent set constraint is an approximation
of the adsorbing potential.  A well-known result of Dobrushin is that
the hard-core model has a phase transition on
$\Z^{2}$~\cite{dobrushin1968problem}.  When $\lambda$ is small, the
adsorbed atoms are disordered, and their spatial correlations decay
exponentially fast.  When $\lambda$ is large, however, they inherit
the periodic order of the crystal substrate --- atoms preferentially
occupy either the even or odd sub-lattices of $\Z^{2}$, and spatial
correlations do not decay. This parity-breaking phase transition is
also encoded in the set of (infinite-volume) Gibbs measures.  There is a
unique Gibbs measure when $\lambda$ is sufficiently small, but
uniqueness fails when $\lambda$ is sufficiently large.  Dobrushin's
proof of non-uniqueness is based on a Peierls argument: he shows that
given a box $\Lambda\subset \Z^{2}$, different choices of boundary
conditions on $\Lambda$ can have an impact on the marginal law of
occupation at the origin, no matter how large $\Lambda$ is.

In this paper we are interested in what happens if the crystal
substrate has some defects. Given $p\in [0,1]$, let $\P_{p}$ denote
the law of Bernoulli site percolation on $\Z^{2}$. That is,
$\P_{p}[X_{v}=1]=p$, $\P_{p}[X_{v}=0]=1-p$, and $X_{v}=1$ indicates
that the vertex $v$ is present. This determines a random subgraph
$G_{p}$ of $\Z^{2}$ by deleting all absent vertices (and any
edge containing an absent vertex).  If $p\approx 1$ only a small
density of vertices are deleted, and intuitively $G_{p}$ is rather
similar to $\Z^{2}$. Perhaps surprisingly, then, our main result is
that there is no phase transition for the hard-core model on $G_{p}$
for \emph{any} $p<1$.
More formally, let $\mathcal{G}_{\lambda}(G)$ denote the set of
infinite-volume Gibbs measures for the hard-core model with activity
$\lambda$ on an infinite graph $G$.
\begin{theorem}
	\label{thm:main}
	Fix $\lambda\geq 0$, $p\in [0,1)$, and let $G_{p}$ be the random
	subgraph of $\Z^{2}$ determined by Bernoulli site percolation on
	$\Z^{2}$. Then $\mathcal{G}_{\lambda}(G_{p})$ is almost surely a
	singleton set.
\end{theorem}

Theorem~\ref{thm:main} is a special case of our main result.
Our more general context allows for random, site-dependent activities
$\bla=\{\lambda_v\}_{v\in V}=\{\lambda X_{v}\}_{v\in V}$, where
$\lambda\geq 0$ is a fixed scaling and the $X_{v}$ are random.  In
this setting the probability of an independent set $I$ is proportional
to $\prod_{v\in I}\lambda_{v}$.

\begin{theorem}
	\label{thm:main-full}
	Fix $\lambda\geq 0$, and suppose the random variables
	$(X_{v})_{v \in \ZZ^2}$ are non-negative, i.i.d., non-constant and
	have finite $(2+\varepsilon)$\textsuperscript{th} moment for some
	$\varepsilon>0$. Then
	$\mathcal{G}_{{\boldsymbol \lambda}}(\mathbb{Z}^2)$ is almost surely
	a singleton~set.
\end{theorem}

Theorems~\ref{thm:main} and~\ref{thm:main-full} may be surprising at first glance, but they are in fact relatively intuitive.
The fundamental observation is one made by Imry and Ma in the context of the random-field Ising model~\cite[Section~7.1]{Bovier06}.
In the setting of Theorem~\ref{thm:main}, their observation is as follows.
Let $\Lambda$ be a finite box in $\Z^{2}$ with even side length, and let $N_\Lambda$ be the difference in the number of even and odd sites removed by the percolation process.
Since we are considering Bernoulli site percolation and $\Lambda$ has even side length, $N_\Lambda$ is mean zero.
Moreover, $N_{\Lambda}$ has (approximately) Gaussian fluctuations of size $\sqrt{|\Lambda|}$.
Since for $\lambda>0$ a positive fraction of vertices are occupied by particles, we expect the fluctuations in $N_{\Lambda}$ to translate into a shift in the size of the partition function by an amount exponential in $\sqrt{|\Lambda|}$ times an \emph{unbounded} random constant; the constant is unbounded since a Gaussian random variable is unbounded.
This shift preferentially puts weight on predominantly even or predominantly odd configurations (which of these two depends on the sign of $N_\Lambda$).
In contrast, the effect of imposing boundary conditions on $\Lambda$ will only have an effect of size $O(\sqrt{|\Lambda|})$, with the implicit constant uniform in $\Lambda$.
Thus there is a positive probability that any effect of boundary conditions will be negligible compared to the effect of the percolation environment in $\Lambda$.
This suggests there will be a unique infinite-volume measure, and that there are no long-range correlations.

While the above heuristic is relatively convincing, it is far from a proof.
A mathematically complete version of this picture for spin systems was developed by Aizenman and Wehr~\cite{AizenmanWehr1990}.
For a textbook exposition, see~\cite{Bovier06}.
Our proof of Theorem~\ref{thm:main-full} adapts the Aizenman-Wehr argument.
The need for adaptation is due to the fact that we are considering a particle system with hard constraints and spatial symmetries, as opposed to a spin system with internal symmetries.

Some elaboration on the last sentence may be useful. That is, one may
wonder why Theorem~\ref{thm:main-full} does not follow from existing
Aizenman-Wehr type arguments. To explain this, first note that the
hard constraint that no two neighbouring sites are simultaneously
occupied prevents any immediate application of the results
in~\cite{AizenmanWehr1990}. One might try to circumvent this by
looking at the marginal distribution of particles on even vertices,
but this leads to issues with translation invariance. Recent
quantitative work on the Imry-Ma phenomenon has done away with the
assumption of starting with a translation invariant
Hamiltonian~\cite{dario2024quantitative}, but this work only considers
Gaussian disorder -- in particular, one would not be able to obtain
Theorem~\ref{thm:main} by directly
applying results from~\cite{dario2024quantitative}.

It is natural to envision other routes to Theorem~\ref{thm:main-full},
e.g., by generalising~\cite{AizenmanWehr1990} to allow for hard
constraints and/or by relaxing their translation invariance
hypothesis. Alternatively, one might aim to
generalise~\cite{dario2024quantitative} to allow for non-Gaussian
disorder. We have opted to avoid generality in favour of a
comparatively brief and simple argument that highlights the phenomenon
of interest: that disorder can destroy the spatial symmetry breaking
phase transition for two-dimensional particle systems.

\subsection{Future Directions and Broader Context}

As suggested by the previous section, our proof of
Theorem~\ref{thm:main-full} follows established lines. The
Aizenman-Wehr method is, however, somewhat delicate, and it is rather
fortunate that it can be adapted to establish our main results. To
highlight this point, we remark that it does \emph{not} appear to be
straightforward to establish the absence of a phase transition for the
hard-core model on a \emph{bond} percolated version of $\Z^{2}$. A
more robust understanding of the effect of mean-zero disorder on phase
transitions for two-dimensional particle systems seems desirable.

The random-field Ising model, and random-field spin systems more
generally, have recently experienced a renaissance. There has been
spectacular quantitative progress regarding the decay of correlations
in the random-field Ising model, first at zero
temperature~\cite{ding2021exponential} and subsequently at positive
temperatures~\cite{ding2021exponential,aizenman2020exponential}. Related
references
include~\cite{ding2023correlation,dario2024quantitative}. Extending
this quantitative understanding to the context of the hard-core model
is a natural question, particularly in light of the fact that the
hard-core model on $\Z^{2}$ serves as a reasonable model for
real-world surface adsorption, where some disorder in the crystal
substrate must be present~\cite{Taylor1985}. Understanding the effect
of more general disorder (e.g., disorder that affects the
bipartite structure of the underlying graph) would also be of
interest.

Given the analogy with the random-field Ising model, it is natural to
expect that the hard-core model retains its phase transition on
$\Z^{d}$, $d\geq 3$. It seems likely that the arguments
of~\cite{ding2024long,ding2024long2} can be adapted to show
this. Another approach would be to verify the abstract conditions
developed in~\cite{chen2025stabilitylongrangeorderdisordered}.

The study of the hard-core model on general bipartite graphs is a
question of significant interest in theoretical computer
science~\cite{dyer2004relative,CannonHelmuthPerkins2024}. Roughly
speaking, the main question is whether or not one can efficiently
generate approximately correct samples from the hard-core model on
general bipartite graphs; for further details and references
see~\cite[Section~1.5]{CannonHelmuthPerkins2024}. While
Theorem~\ref{thm:main-full} has no direct bearing on this question, it
suggests that devising a general-purpose algorithm might be a subtle
matter: the samples produced for $\Z^{2}$ and a sparsely percolated
version of $\Z^{2}$ must be rather different when $\lambda$ is large,
despite the graphs being rather similar. A similar algorithmic
challenge is presented by the Ising model in general external fields,
see~\cite{alaoui2023fast,helmuth2023approximation}. For a formal
connection, see~\cite{goldberg2007complexity}.

There has recently also been interest in the
hard-core model on disordered graphs on
percolated hypercubes and expander graphs. This was initiated in~\cite{KronenbergSpinka2022}; subsequent work includes~\cite{chowdhury2025decouplingclustersindependentsets,geisler2025countingindependentsetspercolated,jenssen2024refinedgraphcontainerlemma}.

\subsection{Acknowledgements}
We thank Ron Peled for helpful comments which led to improvements on an earlier version of this paper.
This research was partly undertaken during the Trimester Program ``Probabilistic methods in quantum field theory'' at the Hausdorff Institute for Mathematics, funded by the Deutsche Forschungsgemeinschaft: EXC-2047/1, 390685813.
The research of IAV, LC and EP is supported by UKRI Future Leaders Fellowship MR/W008513.

\section{Preliminaries}
\label{sec:background}

\subsection{Notation and basic definitions}
We slightly abuse notation by writing $\mathbb{Z}^{2}$ for the graph
$G=(\mathbb{Z}^2,E(\mathbb{Z}^2))$ where $E(\mathbb{Z}^{2})$ consists
of pairs of vertices $u,v \in \mathbb{Z}^2$ such that $\|u-v\|_1=1$. Recall that $\ZZ^2$
is a bipartite graph: its vertices can be disjointly partitioned into
\textbf{even} vertices $e = \{(x,y) \in \ZZ^2 : x+y \equiv 0 \pmod{2}\} $ and \textbf{odd} vertices
$o= \{(x,y) \in \ZZ^2 : x+y \equiv 1 \pmod{2}\}$.
For any $\Lambda \subset \ZZ^2$, we define
$ \partial \Lambda := \{ u \in \ZZ^2 \setminus \Lambda : \exists\, v
	\in \Lambda \text{ with } \|u-v\|_1=1 \,\}$ to be its
\textbf{(external) boundary}.
We almost exclusively focus on even side-length boxes $\Lambda_j := [-j+1,j]^2$, with $j \in \NN$, as these boxes have useful symmetries.

A set of vertices $I \subseteq \ZZ^2$ is called an \textbf{independent
	set} if no two vertices in $I$ are adjacent, i.e., for all $u, v \in
	I$, $u \not\sim v$. See Figure~\ref{fig:IS}.
For $\Lambda \subseteq \ZZ^2$, we denote by $\II_\Lambda$ the
collection of all independent subsets of $\Lambda$. Each independent set $\tau\in \II_{\Z^2}$ can be used to define
\textbf{boundary conditions} by setting
\[
	\II_{\Lambda}^{\tau} := \{ I \in \II_{\ZZ^2} \, : \, I \cap \Lambda^c= \tau \cap \Lambda^c \} \,,
\]
to be the set of all independent subsets in $\Lambda$ that are compatible with $\tau$.
The most important boundary conditions are the \textbf{even and
	odd boundary conditions}, corresponding
to $\tau=e$ and $\tau=o$ respectively.

Let $\bla = (\lambda_v)_{v \in \ZZ^2}$ with $\lambda_v \geq 0$ for each
site $v \in \ZZ^2$. The parameter $\lambda_{v}$ is the \textbf{activity}
at $v$. The \textbf{hard-core model} on $\Lambda$ with boundary condition
$\tau$ is
\begin{equation}
	\label{eq:model}
	\mu^{\tau}_{\Lambda,\bla}(I)
	:=
	\frac{\ind{I \in \II_{\Lambda}^{\tau} }}{Z^{\tau}_{\Lambda,\bla}}
	\prod_{v \in I\cap\Lambda } \lambda_v
	:= \frac{\ind{I \in
			\II_{\Lambda}^{\tau} }}{Z^{\tau}_{\Lambda,\bla}} \bla^{I\cap\Lambda} \,,
\end{equation}
the final equality by the shorthand
$\bla^{A}:=\prod_{v\in A}\lambda_{v}$.
The \textbf{partition function} $Z^{\tau}_{\Lambda,\bla}:=\sum_{I\in
		\II_{\Lambda}^\tau}{\bla^{I\cap \Lambda}}$
ensures that $\mu^\tau_{\Lambda,\bla}$ is a probability measure.
Expectation with respect to $\mu^{\tau}_{\Lambda,\bla}$ will be
denoted by
\begin{equation}
	\bra{F}_{\Lambda,\bla}^{\tau}:=\mu^{\tau}_{\Lambda,\bla}(F) \,.
\end{equation}
In this work we are primarily interested in the case where the
activities are random variables.
That is, we will consider a family of activities $\bla$ given by
\begin{equation}
	\label{eq:defactivities}
	\lambda_v = \lambda \cdot X_v, \qquad v\in \Z^{2},
\end{equation}
where $\lambda > 0$ and $X:=(X_v)_{v \in  \mathbb{Z}^2}$ is a family of i.i.d.\ non-negative random variables with individual laws $\mathbf{P}$ and joint law $\PP$.

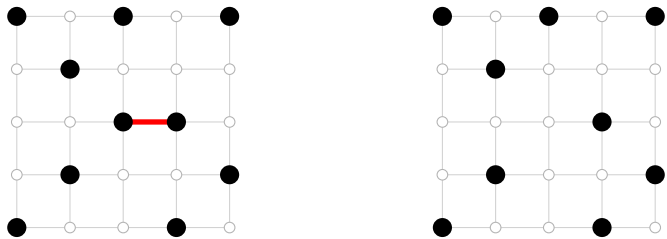
\begin{figure}\centering
	\begin{center}
		\begin{tikzpicture}[scale=0.7]
			\def\rEmpty{0.10} \def\rOcc{0.18}

			\begin{scope}[xshift=0cm]
				\draw[thin, gray!35] (0,0) grid (4,4);

				\draw[red, line width=2pt] (2,2) -- (3,2);

				\foreach \i in {0,...,4}{
						\foreach \j in {0,...,4}{
								\draw[gray!60, fill=white] (\i,\j) circle (\rEmpty);
							}
					}

				\foreach \x/\y in {0/0, 3/0, 1/1, 4/1, 3/2, 1/3, 0/4, 2/4, 4/4}{
						\fill[black] (\x,\y) circle (\rOcc);
					}
				\fill[black] (2,2) circle (\rOcc);
			\end{scope}

			\begin{scope}[xshift=8cm]
				\draw[thin, gray!35] (0,0) grid (4,4);

				\foreach \i in {0,...,4}{
						\foreach \j in {0,...,4}{
								\draw[gray!60, fill=white] (\i,\j) circle (\rEmpty);
							}
					}

				\foreach \x/\y in {0/0, 3/0, 1/1, 4/1, 3/2, 1/3, 0/4, 2/4, 4/4}{
						\fill[black] (\x,\y) circle (\rOcc);
					}
			\end{scope}
		\end{tikzpicture}
	\end{center}   \caption{
		The set of occupied vertices (filled circles) on the left is not
		independent: the two vertices contained in the red edge are
		adjacent. The set on the right is independent.}
	\label{fig:IS}
\end{figure}

\subsection{Infinite-volume Gibbs measures on bipartite graphs}

The hard-core model on bipartite graphs possess
monotonicity properties that simplify the structure of the set of
infinite-volume Gibbs measures.
Given $\bla$,
recall that the set
$\mathcal{G}_{\bla}(\mathbb{Z}^2)$ of Gibbs measures is defined as the
set of measures $\mu$ on
$\mathcal{I}_{\ZZ^2}$ (equipped with the product
$\sigma$-algebra) satisfying the DLR conditions. That is, for any
finite subset $\Lambda\subset \ZZ^2$, the
$\mu$-conditional law of $I|_{\Lambda}$ given
$I|_{\Lambda^c}$ is given by
$\mu^{\tau}_{\Lambda,\bla}$ with
$\tau|_{\Lambda^c}=I|_{\Lambda^c}$ and
$\tau|_{\Lambda}=\emptyset$.  For more details, see, for instance,
\cite[Definition 1.23]{Georgii1988}.

The next theorem characterizes the existence of a phase transition for
the hard-core model on $\mathbb{Z}^2$.
We state the theorem for a deterministic set of activities
$\bla$
and will later apply it for $\lambda_v=\lambda X_v$ with $X$ a fixed
realisation of the random field $X$ under $\mathbb{P}$.

\begin{theorem}
	\label{th:InfiniteVol-PhaseTrans}
	Consider the hard-core model
	on $\ZZ^2$ with activity field $\bla$,
	$0\le \lambda_v<\infty$ for all $v$.
	The even and odd infinite-volume measures
	are well-defined as the (unique) local weak limits of finite-volume measures along any exhausting sequence $\Lambda \uparrow \ZZ^2$:
	\begin{equation}
		\mu_{\ZZ^2,\bla}^{e} = \lim_{\Lambda \uparrow \ZZ^2} \mu_{\Lambda, \bla}^{e} \,, \quad \quad
		\mu_{\ZZ^2,\bla}^{o} = \lim_{\Lambda \uparrow \ZZ^2} \mu_{\Lambda, \bla}^{o}. \vspace{-.2cm}
	\end{equation}
	Furthermore,
	\begin{enumerate}
		\item[(i)]
		      If $\mu_{\ZZ^2, \bla}^{o}(v\in I) = \mu_{\ZZ^2, \bla}^{e}(v\in I)$ for every $v\in \ZZ^2$, then $\mathcal{G}_{\bla}(\mathbb{Z}^2)$ is a singleton set.
		\item[(ii)] For any site $v \in \ZZ^2$, $
			      \mu_{\ZZ^2, \bla}^{e}(v \in I) - \mu_{\ZZ^2, \bla}^{o}(v \in I)\ge 0$ if $v$ is even and $\le 0$ if $v$ is odd.
	\end{enumerate}
\end{theorem}
Theorem~\ref{th:InfiniteVol-PhaseTrans}
can be found, for example, in Theorem~4.18
of~\cite{GeorgiiHaeggMaes2001} or Lemma~3.2 of~\cite{BergSteif1999} in
the case $\lambda_v \equiv \lambda >
	0$.  The proof relies on a standard tool:
the hard-core model satisfies the FKG lattice condition
(see~\cite[Proposition 1]{FKG1971}).  As a consequence, the even and odd boundary
conditions are ordered in the sense of stochastic domination
(see~\cite[Corollary~11]{Holley1974}).  These arguments apply to
site-dependent activities, and hence the proofs in the references
above extend to the site-dependent $\bla$ considered here.

The following \textbf{translation covariance property} for
$\mu_{\ZZ^2,\bla}^{\tau}$ $\tau \in \{o,e\}$ will be important.
For a translation $T_{\ba}$ by $\ba = (a_x, a_y)$,
\begin{equation}
	\label{eq:tranlationCov}
	T_{\ba}\mu^{\tau}_{\ZZ^2,\bla}=\mu_{\ZZ^2,T_{\mathbf{a}}\bla}^{\tau'} \,,
\end{equation}
where $\tau'$ is determined by the parity of the translation:
$\tau'=\tau$ if $a_x+a_y$ is even while
$\tau'$ is the opposite boundary condition ($e \leftrightarrow o$) if $a_x+a_y$ is odd.

The following estimate on occupation probabilities will be useful. By
considering
the extremal case when all neighbours of $v$ are not occupied,
\begin{equation}
	\label{eq:basic_bound}
	{\mu_{\Lambda, \bla}^{\tau} (v \in I)  \leq \frac{\lambda_v}{1+\lambda_v} \,, }
\end{equation}
{which is valid for any $\bla$, $\Lambda$ finite, and
boundary condition $\tau$. By \cref{th:InfiniteVol-PhaseTrans}, this
extends to $\Lambda=\ZZ^2$.}

\section{Proof of the main theorem}
\label{sec:proof-main}
The strategy to prove the theorem, based on the Aizenman-Wehr argument \cite{AizenmanWehr1990}, is to show that:
\begin{itemize}
	\item the {difference in the free energy under different boundary
	      conditions}
	      {due to}
	      the random field $X$ inside a box $\Lambda$ is of order
	      $\sqrt{|\Lambda|}$ times an \emph{unbounded} constant
	      (\cref{lem:gaussian-domination});
	\item the effect of changing boundary conditions {on the free
			      energy} is at most a \emph{deterministic} constant times $|\partial\Lambda|\asymp \hspace{-.1cm}\sqrt{|\Lambda|}$ (\cref{lem:bound-conditional}).
\end{itemize}
These two ingredients taken together lead to a contradiction unless the difference in the free energy between even and odd boundary conditions is zero.
Intuitively, \cref{lem:gaussian-domination} holds since each of the $|\Lambda|$
field variables in $\Lambda$ makes a roughly i.i.d.\ contribution to
the free energy, leading to a Gaussian shift of variance
$|\Lambda|$. \cref{lem:bound-conditional}, in contrast, exploits the
locality of the hard-core constraint: {changing the boundary
conditions can only have a boundary-size effect.}

We now make this more precise. The necessary definitions will be given
in terms of a fixed field $\mathbf{x}=(x_v)_{v\in \ZZ^2}$ with each
$x_v\ge 0$. In the statements of our lemmas we will take $\mathbf{x}$
to be the random activity field $X$.
For activities $\bla=(\lambda x_v)_{v\in \ZZ^2}$ with $\lambda>0$ we
use the following notation to emphasise the dependence on the field:
\begin{equation} \label{eq:Xdefs}
	\mu^{\tau,\bf{x}}_{\Lambda,\lambda}=\mu_{\Lambda, \bla}^\tau; \quad \quad   Z^{\tau}_{\Lambda,\lambda}(\bx)=Z^\tau_{\Lambda, \bla}; \quad \text{and} \quad \bra{\cdot}_{\Lambda,\lambda}^{\tau,\mathbf{x}}=\mu_{\Lambda, \bla}^\tau(\cdot).
\end{equation}
We use the same notation for the infinite-volume measures when $\tau\in \{e,o\}$, replacing $\Lambda$ by $\ZZ^2$.
Then, recalling that for $A$ finite $\mathbf{x}^{A}=\prod_{v\in
		A}x_{v}$, {for $\Lambda$ finite} we define
\begin{equation}
	\label{eq:def-G}
	G^\tau_{\Lambda,\lambda} (\mathbf{x})
	:=
	\frac{1}{\lambda} \log \Bbra{\bx^{I \cap \Lambda} }^{\tau,\mathbf{x}_{\Lambda^c}}_{\mathbb{Z}^2,\lambda} \,,
\end{equation}
where $\bx_{\Lambda^c}$ is given by
\begin{equation}
	\label{eq:switch-off}
	\mathbf{x}_{\Lambda^c}
	:=
	\begin{cases}
		1,   & \text{ if } v \in  \Lambda,   \\
		x_v, & \text{ if } v \in  \Lambda^c.
	\end{cases}
\end{equation}

The following technical lemma, whose proof will be given in
\cref{sec:derivatives}, shows how the derivatives of {$G$} are
connected to the marginal occupation probability of a site $v$. These
identities will be crucial.

\begin{lemma}
	\label{lem:dG}
	Suppose $\Lambda\subset\ZZ^{2}$ is finite, {$\tau \in \mathcal{I}_{\ZZ^2}$}, and
	$\mathbf{x}$
	satisfies $x_v\ge 0$ for all $v$.
	Then $G^\tau_{\Lambda,\lambda}$ is differentiable with respect to
	$\log x_v$ for every $v \in {\Lambda}$, with
	\[
		\frac{\partial}{\partial \log x_{v}}G_{\Lambda, \lambda}^\tau(\mathbf{x})
		= \frac{1}{\lambda} \mu_{\ZZ^2,\lambda}^{\tau, \mathbf{x}} (v \in I) \,.
	\]
\end{lemma}

For the statements of the two key lemmas, we define what roughly
corresponds to the difference in free energy contribution from
$(X_v)_{v\in \Lambda}$ under even and odd boundary conditions:
\begin{align}
	\label{eq:defF}
	F_{\Lambda,\lambda}(X)=F_{\Lambda, \lambda}((X_v)_{v\in \Lambda})
	:=
	\mathbb{E}\left[G^{e}_{\Lambda,\lambda}(X)
		-
		G^{o}_{\Lambda,\lambda}(X)
		\mid
		(X_v)_{v\in \Lambda}\right].
\end{align}

\begin{remark}[Important!] Given $\bx$
	we can define $F_{\Lambda,
				\lambda}$ as a function of $\bx|_{\Lambda}=(x_v)_{v\in
				\Lambda}$ (by conditioning on $(X_v)_{v\in \Lambda}=(x_v)_{v\in
				\Lambda}$)\footnote{This conditioning can readily be seen to be
		well-defined. For instance, since the field variables inside
		$\Lambda$ and outside
		$\Lambda$ are independent we can rephrase the conditional
		expectation as a true expectation but replacing the argument
		$X$ of $G$ with the vector equal to $X$ outside
		$\Lambda$ and equal to $\bx$ inside $\Lambda$.}. We write
	\begin{equation}\label{eq:Fdeterm}
		F_{\Lambda,\lambda}(\bx)=F_{\Lambda,\lambda}(\bx|_\Lambda) \,,
	\end{equation}
	for this function. {The definition~\eqref{eq:defF} is
			this function evaluated at $(X_v)_{v\in \Lambda}$}.
\end{remark}
\cref{lem:dG} implies that $F_{\Lambda,\lambda}$ is differentiable everywhere with respect to $\log x_{v}$ for any $v\in \Lambda$, and
\begin{equation}
	\label{eq:derF}
	\frac{\partial}{\partial \log x_{v}} F_{\Lambda,\lambda}(\bx)=\mathbb{E}[\mu_{\mathbb{Z}^2,\lambda}^{e,X}(v\in I)-\mu_{\mathbb{Z}^2,\lambda}^{o,X}(v\in I)
		\mid (X_w)_{w\in \Lambda}=(x_w)_{w\in \Lambda}] .
\end{equation}
By \cref{th:InfiniteVol-PhaseTrans}(ii) we then have that for all $\bx$,
\begin{align}\label{eq:mono-derivatives-F}
	\frac{\partial}{\partial \log x_{v}} F_{\Lambda,\lambda}(\bx)\ge 0 \,\text{ if $v\in e \,,$   \hspace{3mm} and } \hspace{3mm}
	\frac{\partial}{\partial \log x_{v}} F_{\Lambda,\lambda}(\bx)\le 0 \,\text{ if $v\in o$}.
\end{align}

We now state the two key lemmas.

\begin{lemma}\label{lem:bound-conditional}
	Fix $\lambda >0$.
	Then, there exists a deterministic constant $c_\lambda >0$ such that for any $j\in \NN$ and any $\bx=(x_v)_{v\in \ZZ^2}$ non-negative, we have (recall the notation \eqref{eq:Fdeterm})
	\begin{equation}
		\label{eq:boundF}
		|
		F_{\Lambda_j,\lambda}(\bx)
		|
		\le
		c_\lambda |\partial \Lambda_j|. \qquad
	\end{equation}
\end{lemma}

\begin{lemma}\label{lem:gaussian-domination}
	Fix $\lambda>0$.
	For all $ t \ge 0 $, we have that
	\begin{equation}\label{eq:lem-gaussian-domination}
		\liminf_{j \uparrow \infty}
		\mathbb{E}\left[
			\exp\left( \frac{t F_{\Lambda_j,\lambda}(X)}{ \sqrt{|\Lambda_j|}}\right)
			\right]
		\ge
		\exp(t^2 b^2/2),
	\end{equation}
	where
	\begin{equation}\label{eq:b-dominates-expected-value}
		b^2
		\ge
		\frac{1}{2} \mathbb{E}[\mathbb{E}[F_{\Lambda_j,\lambda}(X) \mid X_{\oo}]^2]
		+
		\frac{1}{2} \mathbb{E}[\mathbb{E}[F_{\Lambda_j,\lambda}(X) \mid X_{\oo'}]^2].
	\end{equation}
	for $\oo=(0,0), \oo'=(0,1)$ and for and any $j \ge 0$.
\end{lemma}
In the proof of Lemma~\ref{lem:gaussian-domination} we will see that
the {right-hand}
side of \eqref{eq:b-dominates-expected-value} does not depend on $\Lambda$ (as long as $\{\oo,\oo'\}\subset\Lambda$).
We will prove the preceding lemmas in the next section. The remainder of
this section shows that they are enough to establish \cref{thm:main-full}.

\begin{proof}[Proof of \cref{thm:main-full} given \cref{lem:dG,lem:bound-conditional,lem:gaussian-domination}.]

	For the proof, we fix $\lambda>0$, and $\Lambda=\Lambda_j$ for some
	$j\ge1$. First, observe that $b$ from
	\eqref{eq:b-dominates-expected-value} must be equal to $0$, otherwise
	there is a clear contradiction between
	\cref{eq:boundF,eq:lem-gaussian-domination}. This implies that $
		\mathbb{E}[\mathbb{E}[F_{\Lambda,\lambda}(X) \mid X_{\oo}]^2]
		+
		\mathbb{E}[\mathbb{E}[F_{\Lambda,\lambda}(X) \mid X_{\oo'}]^2]=0$.
	Hence,
	\begin{equation}\label{eq:F-equiv-zero}
		f(X_{\oo}):=\mathbb{E}[F_{\Lambda,\lambda}(X) \mid X_{\oo}] = 0, \quad \mathbf{P} \text{-a.s}
	\end{equation}
	where we recall that $\mathbf{P}$ is the (marginal) law of
	$X_{\oo}$.

	For $X=(X_w)_{w\in \ZZ^2}$, let $X^{v,x}$ denote the vector obtained
	by replacing $X_v$ with $x$, leaving all other $X_{v'}$
	unchanged. Then setting
	\begin{equation}\label{eq:f}
		f(x)=\mathbb{E}[F_{\Lambda, \lambda}(X^{\oo,x})]\,, \quad \text{ for } x\ge 0.
	\end{equation}
	equation~\eqref{eq:F-equiv-zero} says that $f$ is equal to $0$ on the
	support of $\mathbf{P}$.  Since $f$ is differentiable on $\RR$ with respect to $\log x$ by \eqref{eq:derF} and has non-negative derivative by \eqref{eq:mono-derivatives-F}, we deduce that
	\begin{equation}
		\frac{\partial}{\partial \log x} f \equiv 0
		\text{ on the convex hull of the support of $\mathbf{P}$.}
	\end{equation}
	Using the explicit expression \eqref{eq:derF} for the derivative of
	$F_{\Lambda, \lambda}$
	we thus obtain that
	\begin{equation}
		\frac{\partial}{\partial \log x} f(X_{\oo})=   \mathbb{E}[\mu_{\mathbb{Z}^2,\lambda}^{e,X}(\oo\in I)-\mu_{\mathbb{Z}^2,\lambda}^{o,X}(\oo\in I) \mid X_{\bf o}] = 0 , \quad \mathbf{P} \text{-a.s.}
	\end{equation}
	Taking the expectation of the preceding equation we deduce that
	$
		\mu_{\mathbb{Z}^2,\lambda}^{e,X}(\oo\in I)-\mu_{\mathbb{Z}^2,\lambda}^{o,X}(\oo\in I)
	$
	has zero expectation. Since this is the expectation of a non-negative
	quantity by
	\cref{th:InfiniteVol-PhaseTrans}(ii), we have
	\begin{equation*}
		\mu_{\mathbb{Z}^2,\lambda}^{e,X}(\oo\in I)-\mu_{\mathbb{Z}^2,\lambda}^{o,X}(\oo\in I) = 0, \quad \mathbb{P}\text{-a.s.}
	\end{equation*}
	The same holds for $\oo'$ instead of $\oo$ by the same argument with opposite inequalities, and then, for any $v\in \ZZ^2$ by translation covariance.
	Thus
	$ \mu_{\mathbb{Z}^2,\lambda}^{e,X}(v\in I)-\mu_{\mathbb{Z}^2,\lambda}^{o,X}(v\in I) = 0 $, $\mathbb{P}$-a.s.~for all $v$.
	By \cref{th:InfiniteVol-PhaseTrans}(i), we have that
	$\mathbb{P}$-a.s.,~the set $G_{\bla}(\ZZ^2)$
	is a singleton.
\end{proof}

\section{Proofs of the main lemmas}
In the following sections we prove Lemmas~\ref{lem:dG}, \ref{lem:bound-conditional},
and~\ref{lem:gaussian-domination}, respectively.

\subsection{Occupation probabilities by differentiation}
\label{sec:derivatives}
In this section we prove Lemma~\ref{lem:dG}. This is a technical lemma
that justifies the exchange of differentiation with taking an
infinite-volume limit.
\begin{proof}[Proof of Lemma \ref{lem:dG}]
	Fix $\lambda>0$, $\Lambda$ finite and $\bx = (x_v)_{v \in \ZZ^2}$, with $x_v \ge 0$.
	First, since $\bx$ and $\bx_{\Lambda^c}$ only differ in $\Lambda$, if
	$I \in {\II_{\Lambda_L}^{\tau}}$ then
	\[
		\frac{d \mu_{\Lambda_L, \lambda}^{\tau, \bx}}{d \mu_{\Lambda_L, \lambda}^{\tau, \bx_{\Lambda^c}}} (I)
		=
		\frac{Z_{L,\lambda}^{\tau}(\bx_{\Lambda^c})}{Z_{L,\lambda}^{\tau}(\bx)} \, \bx^{I\cap \Lambda}
		=
		\frac{\bx^{I \cap \Lambda}}{\Bbra{\bx^{I \cap \Lambda}}^{\tau,\bx_{\Lambda^c}}_{\Lambda_L,\lambda} } \,.
	\]
	The last term above only depends on the occupation variables in the finite set $\Lambda$.
	Hence, for {any function $f$ depending only on occupation variables
			in $\Lambda$}
	and every $L$ such that $\Lambda_L\supset \Lambda$ we have
	\begin{equation}
		\label{eq:RN-expectation}
		\Bbra{f}^{\tau, \bx}_{\Lambda_L,\lambda}
		=
		\frac{\Bbra{f \, \bx^{I \cap \Lambda} }^{\tau,\bx_{\Lambda^c}}_{\Lambda_L,\lambda}}{\Bbra{ \bx^{I \cap \Lambda} }^{\tau,\bx_{\Lambda^c}}_{\Lambda_L,\lambda} } \,.
	\end{equation}
	Taking the limit $L\to\infty$
	and using local weak convergence of the finite-volume Gibbs measures, the same relation holds for the infinite-volume Gibbs measures $\mu_{\ZZ^2,\lambda}^{\tau,\bx}$ and $\mu_{\ZZ^2,\lambda}^{\tau,\bx_{\Lambda^c}}$ for~$\tau\in \{e,o\}$.

	Now, we turn to checking the differentiability of $G$.
	Fix a site $v \in \Lambda$ and let $y=\log (x_{v})$.
	Notice that $\bx^{I\cap \Lambda} = \prod_{v \in \Lambda} x_{v}^{\ind{v \in I}}$ is differentiable with respect to $y$ and its derivative is bounded:
	\[
		\frac{\partial}{\partial y} \prod_{w \in \Lambda} x_w^{\ind{w \in I}}
		=
		\ind{v \in I} \prod_{w \in \Lambda} x_w^{\ind{w \in I}}
		\leq
		\prod_{w \in \Lambda} (x_w \vee 1) \,.
	\]
	Since this bound depends only on the finite set $\Lambda$,
	we may apply dominated convergence to exchange derivative and expectation and obtain
	\[
		\frac{\partial}{\partial y} G^\tau_{\Lambda,\lambda} (\mathbf{x})
		=
		\frac{1}{\lambda}
		\frac{\frac{\partial}{\partial y} \Bbra{ \bx^{I \cap \Lambda} }^{\tau,\bx_{\Lambda_j^c}}_{\ZZ^2,\lambda}}{\Bbra{\bx^{I \cap \Lambda}}^{\tau,\bx_{\Lambda^c}}_{\ZZ^2,\lambda}}
		=
		\frac{1}{\lambda}
		\frac{\Bbra{ \ind{v \in I} \,\, \bx^{I \cap \Lambda}
			}^{\tau,\bx_{\Lambda_j^c}}_{\ZZ^2,\lambda}}{\Bbra{\bx^{I \cap
					\Lambda}}^{\tau,\bx_{\Lambda^c}}_{\ZZ^2,\lambda}}
		{=
			\frac{1}{\lambda} \Bbra{ \ind{v \in I}
			}^{\tau,\bx}_{\ZZ^2,\lambda}}\,,
	\]
	The last equality follows from the (infinite-volume limit of the)
	Radon-Nikodym
	derivative~\eqref{eq:RN-expectation}.
\end{proof}

\subsection{Controlling the effect of boundary conditions}

We begin with some remarks.

\subsubsection*{Isometries of $\ZZ^2$ and symmetries between boundary
	conditions}

{The function $F_{\Lambda,\lambda}$ in
	Lemma~\ref{lem:bound-conditional} is
	the difference of the hard-core expectation $G_{\Lambda,\lambda}$ evaluated under opposite boundary conditions (even/odd).
	The conclusion of the lemma is that the magnitude of this difference is controlled by the size of the boundary $|\partial\Lambda|$.
	To prove this, we will perform a transformation that changes the boundary conditions from one to the other, and show that the effect is indeed bounded above by order $|\partial \Lambda|$.
	In spin systems, such as the Ising model, a natural way to do this is through a spin-flip symmetry ($+1\leftrightarrow -1$).
	The hard-core mode, however, lacks an analogous internal symmetry; instead, one can rely on some spatial symmetries of the lattice.
	In particular, we will make use of a
	\emph{vertical reflection} across the line
	$x=1/2$, defined by $\theta ((v_1,v_2)):= (1-v_1,v_2)$.
	This transformation preserves independent sets, exchanges parity, and
	maps every centered box $\Lambda_j$ onto itself. In particular,
	$\theta (\II_{\Lambda_j}^{e}) = \II_{\Lambda_j}^{o}$.
	The proof will
	in fact use a slightly different transformation, but the intuition is
	the same.}

\subsubsection*{Finite volume observables}
The function $F_{\Lambda,\lambda}(X)$ in \eqref{eq:defF} is defined in terms of infinite-volume measures.
However, it will be convenient to perform the main computations of the proof in finite volume and, afterwards, take the thermodynamic limit to obtain the desired result.
We define the finite-volume analogue of \eqref{eq:def-G}, for $L\in
	\NN$, $ \Lambda_L\supset \Lambda$ and $\bx=(x_v)_{v\in \ZZ^2}$
pointwise non-negative by
\begin{figure}
	\centering

	\begin{center}
		\begin{tikzpicture}[scale=0.5]

			\def\rEmpty{0.22}

			\draw[thin, gray] (0,0) grid (9,9);

			\draw[line width=1.4pt, black!80] (0,0) rectangle (9,9);

			\draw[line width=1.4pt, black!80] (3,3) rectangle (6,6);

			\foreach \i in {0,...,9}{
					\foreach \j in {0,...,9}{
							\ifnum\i<3
								\draw[gray!60, fill=gray!20] (\i,\j) circle (\rEmpty);
							\else\ifnum\i>6
									\draw[gray!60, fill=gray!20] (\i,\j) circle (\rEmpty);
								\else\ifnum\j<3
										\draw[gray!60, fill=gray!20] (\i,\j) circle (\rEmpty);
									\else\ifnum\j>6
											\draw[gray!60, fill=gray!20] (\i,\j) circle (\rEmpty);
										\else
											\draw[gray!70, fill=white] (\i,\j) circle (\rEmpty);
										\fi\fi\fi\fi
						}
				}

		\end{tikzpicture}
	\end{center} 	\caption{Illustration of $\bx_{\Lambda_{2}^{c}}$ as defined
		in~\eqref{eq:switch-off}.
		The outer bold square delimits $\Lambda_5$, the inner one
		delimits $\Lambda_2$.
		Sites
		$v \in \Lambda_5 \setminus \Lambda_2$ (in grey) have activity
		$\lambda_v = \lambda \, x_v$, while $\lambda_v = \lambda$
		for $v \in \Lambda_2$.}
	\label{fig:Boxes}
\end{figure}
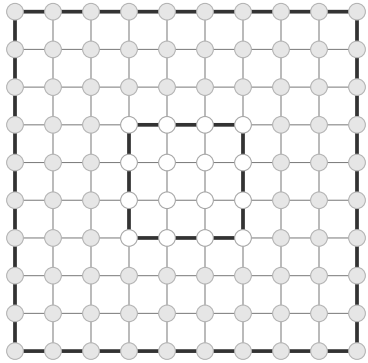
\begin{equation}
	\label{eq:defGLj}
	G_{L,\Lambda,\lambda}^{\tau} (\bx) :=
	\frac{1}{\lambda} \log \langle \bx^{I \cap \Lambda}\rangle^{\tau,\bx_{\Lambda^c}}_{\Lambda_L,\lambda} \,,
\end{equation}
where the expectation is taken with respect to the finite-volume
hard-core measure on $\Lambda_L$, with boundary conditions $\tau \in
	\{e,o\}$ and $\bx_{\Lambda^c}$ denotes the field ``switched off"
inside $\Lambda$ as defined in \cref{eq:switch-off}. See
Figure~\ref{fig:Boxes} {for an illustration of the geometric setup
	in this definition (boundary conditions not illustrated).}
Similarly, we define the finite-volume analogue of \cref{eq:defF} by
\begin{align}
	\label{eq:defFL}
	F_{L,\Lambda,\lambda}(\bx)=F_{L,\Lambda,\lambda}(\bx_{\Lambda^c})
	:=
	\mathbb{E}[G^{e}_{L,\Lambda,\lambda}(X)
		-
		G^{o}_{L,\Lambda,\lambda}(X)
		\mid
		(X_v)_{v\in \Lambda}=(x_v)_{v\in \Lambda}] \,.
\end{align}

\begin{proof}[Proof of \cref{lem:bound-conditional}] Recall the notation $\bx^{S}:= \prod_{v\in S} x_v$ for a vector $\bx$, which will be used throughout the proof.
	Before working with the finite-volume observables defined above, we first justify the passage to the thermodynamic limit.
	Since $\Lambda_L \uparrow \ZZ^2$, \cref{th:InfiniteVol-PhaseTrans}
	yields the convergence of the measures
	\[
		\mu^{\tau,\bx_{\Lambda^c}}_{\Lambda_L,\lambda}\to \mu^{\tau,\bx_{\Lambda^c}}_{\ZZ^2,\lambda} \text{ as } L \to \infty \,,
	\]
	for
	any fixed $\Lambda$, $\bx$ non-negative and $\tau\in\{e,o\}.$
	As the
	function $\bx^{I \cap \Lambda}$ is local, for any boundary condition
	$\tau\in \{e,o\}$, $\lambda,\Lambda, \bx$, we therefore have that
	\[
		G_{\Lambda,\lambda}^{\tau}(\bx)=\lim_{L\to\infty}G_{L,\Lambda,\lambda}^{\tau}(\bx).
	\]
	Moreover,  for $\tau\in \{e,o\}$
	and $\lambda, \Lambda, (x_v)_{v\in\Lambda}$ fixed,
	note that $\bx^{I\cap \Lambda} \leq (1\vee \bx)^{\Lambda}$ for any independent set $I$
	and
	$\mu_{\Lambda_L,\lambda}^{\tau, \bx_{\Lambda^c}}(I\cap \Lambda=\emptyset) \leq  \langle \bx^{I \cap \Lambda}\rangle^{\tau,\bx_{\Lambda^c}}_{\Lambda_L,\lambda}$.
	As a consequence,
	\begin{equation}\label{eq:supG}
		\sup\nolimits_{L,\bx|_{\Lambda^c}}|G_{L,\Lambda, \lambda}^{\tau}(\bx)|\le \frac{1}{\lambda}\sum_{v\in \Lambda} \log(1\vee x_v)+\frac{1}{\lambda}\sup\nolimits_{L,\bx_{\Lambda^c}}\, \big|\log\mu_{\Lambda_L,\lambda}^{\tau, \bx_{\Lambda^c}}(I\cap \Lambda=\emptyset)\big| <\infty,
	\end{equation}
	where the supremum on the right is finite by~\eqref{eq:basic_bound}.
	Thus, by bounded convergence
	\[
		\lim_{L\to \infty} F_{L,\Lambda,\lambda}(\bx|_\Lambda)
		=
		\EE\left[G^{e}_{\Lambda,\lambda}(X)-G^{o}_{\Lambda,\lambda}(X)\mid
		(X_v)_{v\in \Lambda}=(x_v)_{v\in \Lambda} \right]
		=F_{\Lambda,\lambda}(\bx|_\Lambda).
	\]

	As a consequence of the above, it suffices to find $c_\lambda<\infty$
	such that
	\begin{equation}
		\label{eq:FboundFV}
		|F_{L,\Lambda_j,\lambda}(\bx)|\le c_\lambda|\partial
		\Lambda_j| \text{ for all } j<L.
	\end{equation}

	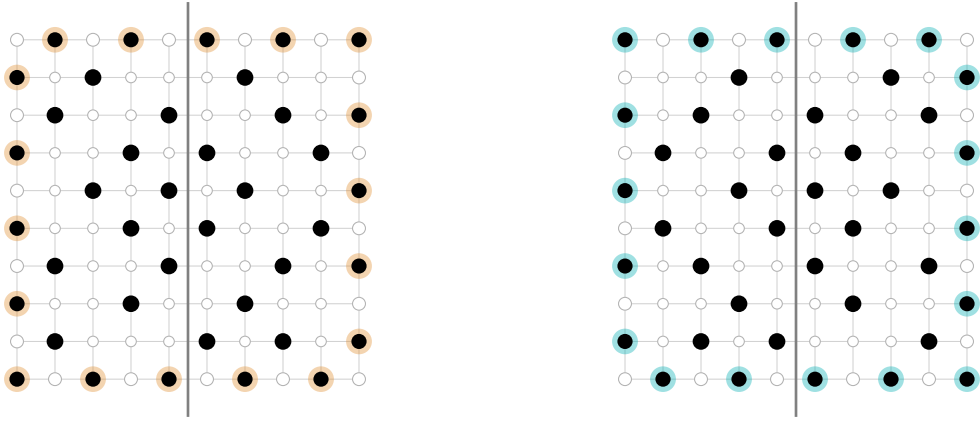
\begin{figure}
		\centering
		\begin{center}
			\begin{tikzpicture}[scale=0.5]
				\def\rEmpty{0.14} \def\rBdry{0.20}  \def\rOcc{0.22}   \def\rBdryEmpty{0.17}

				\colorlet{evenHalo}{BurntOrange!90!black}
				\colorlet{oddHalo}{TealBlue}
				\def\haloScale{1.7}

				\begin{scope}[xshift=0cm]

					\draw[thin, gray!35] (0,0) grid (9,9);
					\draw[line width=1pt, gray] (4.5,-1) -- (4.5,10);

					\foreach \i in {0,...,9}{
							\foreach \j in {0,...,9}{
									\draw[gray!60, fill=white] (\i,\j) circle (\rEmpty);
								}
						}

					\foreach \X in {0,9}{
							\foreach \Y in {1,...,8}{
									\pgfmathtruncatemacro\iseven{mod(\X+\Y,2)==0}
									\ifnum\iseven=1
										\fill[evenHalo, opacity=.38] (\X,\Y) circle (\haloScale*\rBdry);
										\fill[black] (\X,\Y) circle (\rBdry);
									\else
										\draw[gray!60, fill=white] (\X,\Y) circle (\rBdryEmpty);
									\fi
								}
						}
					\foreach \Y in {0,9}{
							\foreach \X in {0,...,9}{
									\pgfmathtruncatemacro\iseven{mod(\X+\Y,2)==0}
									\ifnum\iseven=1
										\fill[evenHalo, opacity=.38] (\X,\Y) circle (\haloScale*\rBdry);
										\fill[black] (\X,\Y) circle (\rBdry);
									\else
										\draw[gray!60, fill=white] (\X,\Y) circle (\rBdryEmpty);
									\fi
								}
						}

					\foreach \x/\y in {
							1/1,1/3,1/7,2/5,2/8,
							3/2,3/4,3/6,4/3,4/5,4/7,
							5/1,5/4,5/6,6/2,6/5,6/8,
							7/1,7/3,7/7,8/4,8/6
						}{
							\fill[black] (\x,\y) circle (\rOcc);
						}
				\end{scope}

				\begin{scope}[xshift=16cm]

					\draw[thin, gray!35] (0,0) grid (9,9);
					\draw[line width=1pt, gray] (4.5,-1) -- (4.5,10);

					\foreach \i in {0,...,9}{
							\foreach \j in {0,...,9}{
									\pgfmathtruncatemacro\xpos{9-\i}
									\draw[gray!60, fill=white] (\xpos,\j) circle (\rEmpty);
								}
						}

					\foreach \X in {0,9}{
							\foreach \Y in {1,...,8}{
									\pgfmathtruncatemacro\isodd{mod(\X+\Y,2)==1}
									\ifnum\isodd=1
										\fill[oddHalo, opacity=.38] (\X,\Y) circle (\haloScale*\rBdry);
										\fill[black] (\X,\Y) circle (\rBdry);
									\else
										\draw[gray!60, fill=white] (\X,\Y) circle (\rBdryEmpty);
									\fi
								}
						}
					\foreach \Y in {0,9}{
							\foreach \X in {0,...,9}{
									\pgfmathtruncatemacro\isodd{mod(\X+\Y,2)==1}
									\ifnum\isodd=1
										\fill[oddHalo, opacity=.38] (\X,\Y) circle (\haloScale*\rBdry);
										\fill[black] (\X,\Y) circle (\rBdry);
									\else
										\draw[gray!60, fill=white] (\X,\Y) circle (\rBdryEmpty);
									\fi
								}
						}

					\foreach \x/\y in {
							1/1,1/3,1/7,2/5,2/8,
							3/2,3/4,3/6,4/3,4/5,4/7,
							5/1,5/4,5/6,6/2,6/5,6/8,
							7/1,7/3,7/7,8/4,8/6
						}{
							\pgfmathtruncatemacro\xpos{9-\x}
							\fill[black] (\xpos,\y) circle (\rOcc);
						}

				\end{scope}
			\end{tikzpicture}
		\end{center} 	\caption{{Two independent sets on $\Lambda_4$ with boundary conditions specified on $\Lambda_5$.
					On the left, the configuration has even boundary conditions (orange).
					On the right, the reflected configuration
					has odd boundary conditions (blue).}}\label{fig:flip}
	\end{figure}

	\noindent To show \eqref{eq:FboundFV}, observe that for any activity field $\mathbf{y}$, by definition,
	\begin{equation}\label{eq:RN}
		G_{L,{\Lambda_{j}},\lambda}^\tau(\mathbf{y})=\frac{1}{\lambda}\big(\!\log{Z^\tau_{L,\lambda}(\mathbf{y})}-\log {Z^\tau_{L,\lambda}(\mathbf{y}_{\Lambda_j^c})}\big) \,,
	\end{equation}
	and hence
	\begin{equation}
		\label{eq:Fexpression}
		F_{L,\Lambda_j,\lambda}(\mathbf{x})=\frac{1}{\lambda} \mathbb{E}\left[\log \frac{Z^{e}_{L,\lambda}(X)Z^{o}_{L,\lambda}(X_{\Lambda_j^c})}{Z^{o}_{L,\lambda}(X)Z^{e}_{L,\lambda}(X_{\Lambda_j^c})} \; \Big| \;  (X_v)_{v\in \Lambda_j}=(x_v)_{v\in \Lambda_j}\right].
	\end{equation}
	Since the activity variables are i.i.d., the conditional expectation above amounts to fixing $X_v=x_v$ for $v\in \Lambda_j$ and taking expectation over $(X_v)_{v\in \Lambda_L\setminus \Lambda_j}$.
		{In the remainder of this paragraph we describe the key idea in estimating
			this expectation; more precise details then follow.} The approach is to apply a transformation $\phi$ that keeps vertices inside $\Lambda_{j+1}$
	fixed, and reflects $\Lambda_L \setminus \Lambda_{j+1}$ across $x=1/2$; see Figure \ref{fig:flip} for an illustration.
	This swaps the parity of boundary conditions on $\partial \Lambda_L$ and yields an almost one-to-one correspondence between independent sets, with potential conflicts confined to the annulus $\Lambda_{j+1}\setminus \Lambda_j$.
	It follows readily that for any non-negative activity $\mathbf{y}$
	and $\{\tau,\tau'\}=\{e,o\}$ or $\{o,e\}$
	\begin{equation}
		\label{eq:PFevenodd1}
		Z_{L,\lambda}^{\tau} (\mathbf{y})
		\leq
		(1 + \lambda \mathbf{y} )^{\Lambda_{j+1}\setminus \Lambda_j} \, Z_{L,\lambda}^{\tau'} (\mathbf{y}\circ \phi
		)
		\,.
	\end{equation}
	When one then plugs this inequality into \eqref{eq:Fexpression}, since
	an expectation is being taken over the i.i.d.~activity variables
	outside $\Lambda_j$ {and the transformed field $X\circ\phi$ has the same distribution as $X$}, all that remains is a sum of $\EE[\log(1+\lambda X_v)]$ over $v\in \Lambda_{j+1}\setminus \Lambda_j$, which is bounded above by a uniform constant times~$|\partial \Lambda_j|$.

	The remainder of the proof formalises the discussion above. Fix $j$ and $L$, and set
	\[\Lambda_L=A_1\cup A_2\cup A_3 \qquad \text{where} \quad
		A_1:= \Lambda_j, \quad \quad
		A_2:= \Lambda_{j+1}\setminus \Lambda_j ,\quad \text{and} \quad
		A_3 := \Lambda_L \setminus \Lambda_{j+1} \,.
	\]
	and for $I \in \II_{L}$ and $k \in \{ 1,2,3 \}$, set \( I_k=I \cap A_k \) and $\II_k:=\II_{A_k}$.
	To show \eqref{eq:PFevenodd1}, we use this  decomposition to write
	\[
		Z_{\Lambda_L,\lambda}^{\tau} (\mathbf{y})
		=
		\sum_{I \in \II_{L}^{\tau}} (\lambda \mathbf{y})^I
		=
		\sum_{I_1 \in \II_{1}} (\lambda \mathbf{y})^{I_1}
		\sum_{I_3 \in \II_3^{\tau}} (\lambda \mathbf{y})^{I_3}
		\sum_{I_2 \in \II_2(I_1,I_3)} (\lambda \mathbf{y})^{I_2} \,,
	\]
	for $\tau\in \{e,o\}$ where $\II_3^{\tau}$ are the independent sets in $A_3$ compatible with boundary conditions $\tau$ in $\Lambda_L^c$ and $\II_2(I_1,I_3)$ are the independent sets in $A_2$ compatible with $I_1$ in $A_1 = \Lambda_j$ and with $I_3$ in $A_3$.

	We consider the bijection $\phi$ of $\ZZ^2$ described above that reflects sites in $A_3$ via $\theta$ and leaves all other sites fixed:
	\begin{equation*}
		\phi(v)=\phi^{(j)}(v):=
		\begin{cases}
			v,         & \text{if } v \in \Lambda_{j+1} \,,      \\
			\theta(v), & \text{if } v \not\in  \Lambda_{j+1} \,.
		\end{cases}
	\end{equation*}
	The map $\phi$ is not an isometry, so it may fail to preserve independent sets, but this issue can only arise from contributions in $A_2$.
	To deal with this, we use a uniform bound that allows us to “erase’’ $I_2$ and avoid possible conflicts.
	Since \(\II_2(I_1,I_3)\subseteq \mathcal P(A_2)\), we have
	\[
		\sum_{I_2 \in \II_2(I_1,I_3)}  (\lambda \mathbf{y})^{I_2}
		\leq \sum_{S \subseteq A_2} (\lambda \mathbf{y})^S
		=
		(1 + \lambda \mathbf{y})^{A_2} ,
	\]
	and so since everything is positive
	\[
		Z_{L,\lambda}^{\tau} (\mathbf{y})
		\leq
		(1 + \lambda \mathbf{y})^{A_2}
		\sum_{I_1 \in \II_{1}} (\lambda \mathbf{y})^{I_1}
		\sum_{I_3 \in \II_3^{\tau}} (\lambda \mathbf{y})^{I_3} \,.
	\]

	Moreover, since $\phi$ induces a bijection between $\II_3^{\tau}$ and $\II_3^{\tau'}$ (where $\tau'(e):=o$ and $\tau'(o):=e$)
	setting $J_3 :=\phi(I_3)$,
	where $J_3 \in \II_3^{\tau'}$ for every $I_3 \in \II_3^{\tau}$,   we get
	\[
		Z_{L,\lambda}^{\tau} (\mathbf{y})
		\leq
		(1 + \lambda \mathbf{y})^{A_2}
		\sum_{I_1 \in \II_{1}} (\lambda \mathbf{y})^{I_1}
		\sum_{J_3 \in \II_3^{o}} \prod_{u \in J_3} \lambda (\mathbf{y}\circ \phi)_u
		\leq
		(1 + \lambda \mathbf{y} )^{A_2} \, Z_{L,\lambda}^{\tau'} (\mathbf{y}\circ \phi
		)
		\,,
	\]
	where $(\mathbf{y}\circ \phi)_v := \mathbf{y}_{\phi(v)}$. This is exactly \eqref{eq:PFevenodd1}.

	Hence, $\log Z^{\tau'}_{L,\lambda}(\mathbf{y}\circ \phi)- \sum_{v\in
			A_2} \log(1+\lambda y_v)\le \log Z_{L,\lambda}^\tau(\mathbf{y})\le
		\log Z^{\tau'}_{L,\lambda}(\mathbf{y}\circ \phi) +\sum_{v\in A_2}
		\log(1+\lambda y_v)$ for every $\mathbf{y}$ non-negative and
	${\tau,\tau'}=\{e,o\}$ or $\{o,e\}$.  Inserting this bound in
	\eqref{eq:PFevenodd1}, and using that, for $\tau\in \{e,o\}$,
	$\mathbb{E}[\log Z_{L,\Lambda}^\tau(X) \mid (X_v)_{v\in
				\Lambda_j}=(x_v)_{v\in \Lambda_j}]=\mathbb{E}[\log
			Z^{\tau}_{L,\lambda}(X\circ \phi)\mid (X_v)_{v\in
				\Lambda_j}=(x_v)_{v\in \Lambda_j}]$ and $\mathbb{E}[\log
			Z_{L,\Lambda}^\tau(X_{\Lambda_j^c}) \mid (X_v)_{v\in
				\Lambda_j}=(x_v)_{v\in \Lambda_j}]=\mathbb{E}[\log
			Z^{\tau}_{L,\lambda}(X_{\Lambda_j^c}\circ \phi)\mid (X_v)_{v\in
				\Lambda_j}=(x_v)_{v\in \Lambda_j}]$, we obtain
	\begin{equation}
		\label{e:Fbound}
		|F_{L,\Lambda_j,\lambda}(\bx)| \le \frac{2}{\lambda} \sum_{v\in A_2} \mathbb{E}[\log(1+\lambda y_v)] \le c'_\lambda
		|A_2|  \,,
	\end{equation}
	with $c'_\lambda =(2/\lambda)\mathbf{E}[\log(1+\lambda X_v)]$; the
	expectation does not depend on the vertex $v$.
	Since
	$|A_2|\le c |\partial \Lambda_j|$ for a universal constant $c$, this concludes the proof.
\end{proof}

\subsection{Gaussian Domination}
In this section, we prove Lemma~\ref{lem:gaussian-domination}.
\begin{proof}[Proof of \cref{lem:gaussian-domination}]
	Order the vertices of $\mathbb{Z}^2$ lexicographically. For $v\in \mathbb{Z}^2$, $\Lambda \subset \ZZ^2$ finite, set
	\[
		\mathcal{H}^{\le v}=\sigma((X_w)_{w\le v}),\quad
		\mathcal{H}^{\le v}_{\Lambda}=\sigma((X_w)_{w\le v, w\in \Lambda}), \quad \text{and} \quad
		\mathcal{H}_\Lambda=\sigma((X_w)_{w\in \Lambda}),
	\]
	{where $\sigma(A)$ denotes the sigma-algebra generated by a set of
	random variables $A$.}
	Define $\mathcal{H}^{<v}$ and $\mathcal{H}^{<v}_{\Lambda}$ analogously
	to $\mathcal{H}^{\le v}$.  To simplify notation in this proof we will
	work with $\Lambda=[0,2j-1]^2$; the case $\Lambda=\Lambda_j$ follows by translation covariance.

	Let $v_1,\dots, v_{|\Lambda|}$ denote the vertices of $\Lambda$ in
	order; note that $v_1=\oo$ and $v_2=(1,0)=\oo'$. For $1\le i \le |\Lambda|$, set
	\begin{align}
		Y_i & :=
		\EE\left[G_{\Lambda,\lambda}^e(X)-G_{\Lambda,\lambda}^o(X)|\mathcal{H}^{\le v_i}_{\Lambda}\right]
		-\EE\left[G_{\Lambda,\lambda}^e(X)-G_{\Lambda,\lambda}^o(X)|\mathcal{H}^{<
		v_i}_{\Lambda}\right] \nonumber                                                                                                                                  \\
		    & = \EE\left[F_{\Lambda,\lambda}|\mathcal{H}_{\Lambda}^{\le v_i}\right]-\EE\left[F_{\Lambda,\lambda}|\mathcal{H}_{\Lambda}^{< v_i}\right] \,, \label{eq:YiF}
	\end{align}
	by the tower law. {For future use, note that
			$F_{\Lambda,\lambda}(X) = \sum_{i=1}^{|\Lambda|} Y_i$.}
	Set
	\[
		W_i:=\EE\left[G_{\Lambda,\lambda}^e(X)-G_{\Lambda,\lambda}^o(X)|\mathcal{H}^{\le v_i}_{}\right]-\EE\left[G_{\Lambda,\lambda}^e(X)-G_{\Lambda,\lambda}^o(X)|\mathcal{H}^{< v_i}_{}\right].
	\]
	Notice that for any vertex $u$, since  $\EE[G_{\Lambda,\lambda}^e(X)-G_{\Lambda,\lambda}^o(X)|\mathcal{H}^{\le
			u}]$ is a function of $(X_w)_{w\le u}$ only, its conditional
	expectation given $(X_w)_{w\in \Lambda}$ is the same as its
	conditional expectation given $(X_w)_{w\in \Lambda, w\le u}$. A
	tower law computation using the definition of $Y_{i}$ thus gives
	\begin{equation}
		\label{eq:YW}
		Y_i = \mathbb{E}[W_i \mid \mathcal{H}_\Lambda].
	\end{equation}

	The key observation about the $W_i$ is the following.
	For $x\in [0,\infty)$, let $X^{v,x}$ denote the field obtained from replacing the value $X_v$ of $X$ at site $v$ by $x$.
	Let $w=v_i$ for some $i$, we have
	\begin{equation}
		\label{eq:IntGe-Go}
		\begin{aligned}
			G_{\Lambda,\lambda}^e(X)-G_{\Lambda,\lambda}^o(X)
			 & = \int_{-\infty}^{\log X_{w}} \frac{\partial}{\partial
				\log({x_w})}
			\left(G_{\Lambda,\lambda}^e\left(X^{w,z}\right)-G_{\Lambda,\lambda}^o\left(X^{w,z}\right) \right) \,dy \qquad \text{where $z=e^y$ } \\
			 & = \int_{-\infty}^{\log X_{w}}
			\mu_{\ZZ^2,\lambda}^{e, X^{w,z}}(w\in I)-\mu_{\ZZ^2,\lambda}^{o, X^{w,z}}(w\in I)  \,dy                                             \\
			 & = \int_{0}^{X_{w{}}}
			\mu_{\ZZ^2,\lambda}^{e, X^{w,x}}(w\in I)-\mu^{o, X^{w,x}}_{\ZZ^2,\lambda}(w\in I) \, \frac{dx}{x}.
		\end{aligned}
	\end{equation}
	By \eqref{eq:basic_bound}
	$\mu_{\ZZ^2,\lambda}^{\tau, X^{w,x}}(w\in I)\le \lambda x$,
	so this integral is well defined near $0$.
	Let
	\[
		g_w ((X_v)_{v<w},x)
		= \EE\left[\mu_{\ZZ^2,\lambda}^{e, X^{w,x}}(w\in I) -
		\mu_{\ZZ^2,\lambda}^{o, X^{w,x}}(w\in I) \,  \Big| \,
		{\mathcal{H}^{<w}}
		\right]\,.
	\]
	In other words, $g_w ((X_v)_{v<w},x)$ is the expectation of the
	difference
	in the occupation probabilities of $w$ when we fix $X_w=x$, given the
	field up until vertex $w$, {under the even and odd measures}.
	With this notation we can express
	\[
		W_i
		= W_i((X_v)_{v<v_i},X_{v_i}))
		= \int_0^{X_{v_i}} g_{v_i}((X_v)_{v<v_i},x) \frac{dx}{x}
		- \mathbf{E} \left[ \int_0^{Z} g_{v_i}((X_v)_{v<v_i},x) \frac{dx}{x} \right] \,,
	\]
	where $Z$ is an independent copy of $X_{v_i}$ and $\mathbf{E}$ denotes its law.
	In particular, notice that $g_{v_i}$ does not depend on $\Lambda$, and so neither does $W_i$.
	Using \eqref{eq:basic_bound} again, we have that $|g_{v_i}((X_v)_{v<v_i},x)| \le 2\lambda x$ for all $v_i, (X_v)_{v<v_i}$ which implies that
	\begin{equation}
		\label{eq:Wbound}
		|W_i|\le 2\lambda (X_{v_i}+\mathbb{E}[X_{v_i}]) \,,
	\end{equation} for all $i$, since $g_{v_i}$ has a fixed sign determined by the parity of $v_i$.
	Moreover, let $v_j = v_i + \ba$, where the vector $\ba$ has even parity.
	Then, by the translation covariance of the expectations inside $g_{v_i}$, we have
	\[
		g_{v_i + \ba} ((X_v)_{v<v_i+\ba}, X_{v_i+\ba})
		=
		g_{v_i} \Big( ((T_{\ba}X)_v)_{v<v_i}, (T_{\ba}X)_{v_i} \Big) \,,
	\]
	where $(T_{\ba}X)_v = X_{T_{\ba}v} = X_{v+ \ba}$.
	Therefore,
	\begin{equation}
		\label{eq:Wshift} W_j(X) = W_i(T_{\ba}X).
	\end{equation}

	The remainder of this proof is divided in several steps.

	\subsubsection*{Step 1}
	We first show that $\EE [F_{\Lambda,\lambda} (X)]=0$ for any $j \ge 0$.
	By \eqref{eq:supG} and our conditions on the law of the field, $G^{\tau}_{\Lambda,\lambda}(X)$ is integrable for $\tau\in \{e,o\}$.
	Hence $\EE[G^{\tau}_{\Lambda,\lambda}(X)]$ is well-defined.
	Let $\ba=(1,0)$.
	By the translation covariance property \eqref{eq:tranlationCov} of the infinite volume measures, we have
	\[
		G^{o}_{\Lambda,\lambda}(X)
		=
		G^{e}_{\Lambda+\ba,\lambda}(T_{\ba}X) \,,
	\]
	where we have used that $T_{\ba}o=e$ for $\ba$ with odd parity.
	Since $T_{\ba}X$ and $X$ have the same distribution we obtain that
	\[
		\EE [G^{o}_{\Lambda,\lambda}(X)]
		=
		\EE [G^{e}_{\Lambda + \ba,\lambda}(T_{\ba}X)]
		=
		\EE[G^{e}_{\Lambda+\ba,\lambda}(X)]\,.
	\]
	To show that $\EE[F_{\Lambda, \lambda}(X)]=0$ it thus suffices to establish
	\[\EE[G^{e}_{\Lambda+\ba,\lambda}(X)]=\EE[G^{e}_{\Lambda,\lambda}(X)]\,.
	\]
	To verify this we use the definition of $G$ and of infinite volume
	Gibbs measures to write
	$\exp(\lambda G^e_{\Lambda}(X))=\lim_{L\to \infty} \langle X^{I\cap
				\Lambda} \rangle^{e, X_{\Lambda^c}}_{B(L,j),\lambda}$, where $B(L,j)$
	is the box $[-2L+j,2L+j]\times [-2L,2L]$. Letting $Y(v)=X(\theta(v))$,
	where $\theta$ is reflection in the line $\{x=j\}\subset \ZZ^2$, we
	also have
	\begin{equation}\label{eq:XY}
		\langle X^{I\cap \Lambda} \rangle^{e, X_{\Lambda^c}}_{B(L,j),\lambda}=\langle Y^{I\cap (\Lambda+\ba)} \rangle^{e, Y_{(\Lambda+\ba)^c}}_{B(L,j),\lambda} \,,
	\end{equation}
	for every fixed $L$. The limits on either side of \eqref{eq:XY} as
	$L\to \infty$ therefore agree. The limit on the right-hand side is
	$\exp(\lambda G_{\Lambda+\ba}^{e}(Y))$. Since $Y$ and $X$ have the
	same law (under $\mathbb{P}$), this implies that $\exp(\lambda
		G_{\Lambda}^e(X))$ has the same law as $\exp(\lambda
		G_{\Lambda+\ba}^e(X))$. In particular, the same holds for
	$G_{\Lambda}^e(X)$ and $G_{\Lambda+\ba}^e(X)$, and the claim concerning their expectations follows.

	\subsubsection*{Step 2}
	Set $F_\Lambda := F_{\Lambda,\lambda}(X) = \sum_{i=1}^{|\Lambda|} Y_i$.
	By~\eqref{eq:YiF} and successively conditioning,
	\[
		\mathbb{E}\left[\,e^{\frac{tF_{\Lambda}}{\sqrt{|\Lambda|}}\,}\right]=\mathbb{E}\left[ \,\prod_{i=1}^{|\Lambda|}
			\mathbb{E}\left[\,
				e^{\frac{tY_i}{\sqrt{\Lambda}}}\Big| \mathcal{H}_\Lambda^{<v_i}
				\right]\,
			\right].
	\]
	{This, when combined with the mean zero property established in
	Step~1, has the following consequence.} For any $a>0$, and for an (explicit) function $f\colon
		[0,\infty)\to [0,1]$ with $f(a)\downarrow 0$ as $a\downarrow 0$ (see~\cite[(7.53)]{Bovier06}),
	\begin{equation}
		\label{eq:lapbound}
		\mathbb{E}\left[
			\,
			e^{\frac{tF_{\Lambda}}{\sqrt{|\Lambda|}}}
			\,
			\right]\ge \EE\left[\exp\left(\frac{t^2(1-f(a))}{2|\Lambda|}\sum_{i=1}^{|\Lambda|}
			\mathbb{E}\left[Y_i^2 \mathds{1}_{t|Y_i|\le
				a\sqrt{|\Lambda|}}\,\Big|\,
				\mathcal{H}_\Lambda^{<v_i}\right]\right)\right].
	\end{equation}
	\subsubsection*{Step 3}
	In this step, we prove that, for any $a>0$,
	\[
		\frac{1}{|\Lambda|}\sum_{i=1}^{|\Lambda|}
		\mathbb{E}\left[Y_i^2 \mathds{1}_{t|Y_i|\ge a\sqrt{|\Lambda|}} \,\Big|\, \mathcal{H}_\Lambda^{<v_i}\right] \stackrel{\PP}{\longrightarrow} 0 \,,
	\]
	as $\Lambda \uparrow \mathbb{Z}^2$.
	Pick $p>1$ such that $\mathbb{E}[X_v^{2p}]<\infty$, and set
	$q$ such that $1/q+1/p=1$. Then
	\begin{align*}
		\EE \left[ \frac{1}{|\Lambda|} \sum_{i=1}^{|\Lambda|}\EE [Y_i^2 \mathds{1}_{t|Y_i|\ge a \sqrt{|\Lambda|}} \Big| \mathcal{H}_\Lambda^{<v_i} ] \right]
		 & =
		\frac{1}{|\Lambda|} \sum_{i=1}^{|\Lambda|} \mathbb{E}[Y_i^2 \mathds{1}_{t|Y_i| \ge a\sqrt{|\Lambda|}}]                 \\
		 & \le
		\frac{1}{|\Lambda|} \sum_{i=1}^{|\Lambda|} \mathbb{E}[Y_i^{2p}]^{1/p} \mathbb{P}({t|Y_i| \ge a\sqrt{|\Lambda|}})^{1/q} \\
		 & \le
		\frac{1}{|\Lambda|} \sum_{i=1}^{|\Lambda|} \mathbb{E}[W_i^{2p}]^{1/p} \left( \frac{t\mathbb{E}[|W_i|]}{a\sqrt{|\Lambda|}}\right)^{1/q} \,,
	\end{align*}
	where in the last inequality we have used \eqref{eq:YW} and that the conditional expectation is a contraction in $L^{r}$ for $r\ge 1$.
	By \eqref{eq:Wbound} and since the variables $X_v$ are i.i.d., we have uniform bounds
	$\EE[|W_i|^{2p}]^{1/p} \leq (2\lambda)^2 (\EE[X_v^{2p}]^{1/p}+\mathbb{E}[X_v]^2)$
	and
	$\EE[|W_i|] \leq 4\lambda \EE[X_v]$ for all $i$.
	Substituting these into the last expression gives a term of order $|\Lambda|^{-1/(2q)}$, which  converges to $0$ as $|\Lambda|\to \infty$.
	\subsubsection*{Step 4}
	In this step, we apply an appropriate ergodic theorem to show that
	\[
		\frac{1}{|\Lambda|} \sum_{i=1}^{|\Lambda|} \EE [W_i^2 \mid \mathcal{H}^{<v_i}]
		\stackrel{\PP}{\longrightarrow}
		\frac12 \left(\mathbb{E}[W_{1}^2]+\mathbb{E}[W_{2}^2]\right)=:b^2 \text{ as } \Lambda \uparrow \ZZ^2 \,,
	\]
	and show that $b^2$ satisfies \eqref{eq:b-dominates-expected-value}.

	Recall that $\oo=(0,0)=v_1$ and $\oo'=(0,1)=v_2$. Let
	$f(X)=\mathbb{E}[W_1^2 |\mathcal{H}^{<v_1}]$. By \eqref{eq:Wshift}, for $v_i$ even,
	\[ \mathbb{E}[W_i^2|\mathcal{H}^{<v_i}] = f(T_{v_i}(X)),\]
	and similarly for $v_i$ odd.
	Therefore, we can (using that $(X_v)_v$ are i.i.d.) apply the Tempel'man Ergodic Theorem (\cite[Theorem 2.8]{krengel2011ergodic}, for instance) separately to even and odd sites to obtain that
	\begin{equation}
		\begin{aligned}
			\frac{1}{|\Lambda|} \sum_{i=1}^{|\Lambda|} \EE [W_i^2 \mid \mathcal{H}^{<v_i}]
			 & =
			\frac{1}{|\Lambda|} \sum_{v_i \in \Lambda \cap e } \EE [W_i^2 \mid \mathcal{H}^{<v_i}]
			+
			\frac{1}{|\Lambda|} \sum_{v_i \in \Lambda \cap o} \EE [W_i^2 \mid \mathcal{H}^{<v_i}] \\
			 & \stackrel{\PP}{\longrightarrow}
			\frac12 \left(\mathbb{E}[W_{1}^2]+\mathbb{E}[W_{2}^2]\right)=:b^2 \text{ as } \Lambda \uparrow \ZZ^2 \,.
		\end{aligned}
	\end{equation}
	Moreover, by (conditional) Jensen's inequality, we have
	\[
		\mathbb{E}[W_1^2]\ge \mathbb{E}[\mathbb{E}[W_{1}|X_{\oo}]^2]=\mathbb{E}[\mathbb{E}[F_{\Lambda}|X_{\oo}]^2] \,,
	\]
	where the equality follows since (omitting subscripts and arguments of
	$G$)
	\[\mathbb{E}[W_1|X_{\oo}]=\mathbb{E}[\EE[G^e-G^o|\mathcal{H}^{\le
				\oo}]|X_{\oo}]=\mathbb{E}[G^e-G^o | X_{\oo}]=\EE[ \EE[ G^e-G^o |
			\mathcal{H}_\Lambda]|X_{\oo}]=\mathbb{E}[F_\Lambda|X_{\oo}] \,, \]
	where we have used the tower law. {To obtain the desired lower bound on $b^{2}$ we will
	argue similarly to show $\mathbb{E}[W_2^2]\ge
		\mathbb{E}[\mathbb{E}[F_\Lambda|X_{\oo'}]^2]$. Indeed, since
	$X_{\oo'}$ is independent of $\mathcal{H}^{\le \oo}$,
	\begin{align*}
		\mathbb{E}[W_2|X_{\oo'}]
		 & =\mathbb{E}[\EE[G^e-G^o|\mathcal{H}^{\le \oo'}]|X_{\oo'}]-\mathbb{E}[\EE[G^e-G^o|\mathcal{H}^{\le \oo}]|X_{\oo'}] \\
		 & =\mathbb{E}[G^e-G^o | X_{\oo'}]-\mathbb{E}[G^{e}-G^{o}]                                                           \\
		 & =\EE[ \EE[ G^e-G^o | \mathcal{H}_\Lambda]|X_{\oo'}]=\mathbb{E}[F_\Lambda|X_{\oo'}].
	\end{align*}
	as the second term in the second line is zero by Step 1.}

	\subsubsection*{Step 5}
	In this (technical) step we show that
	\begin{equation}\label{eq:step5}
		\frac{1}{|\Lambda|}\sum_{i=1}^{|\Lambda|}\left(\mathbb{E}[Y_i^2 \mid \mathcal{H}_\Lambda^{<v_i}]-\mathbb{E}[W_i^2\mid \mathcal{H}^{<v_i}]\right)
		\stackrel{\PP}{\longrightarrow} 0 \text{ as } \Lambda \uparrow \ZZ^2 \,.
	\end{equation}
	First, since $Y_i= \EE[W_i | \mathcal{H}_{\Lambda}]$ by~\eqref{eq:YW} and $\mathcal{H}^{<v_i}_{\Lambda} \subset \mathcal{H}^{<v_i}$, we have
	\begin{align}
		\EE \Big[\, \Big| \, \EE[Y_i^2 \mid \mathcal{H}_\Lambda^{<v_i}] -\EE[W_i^2\mid \mathcal{H}^{<v_i}]\, \Big| \,\Big]
		 & =
		\EE \Big[\, \Big| \, \mathbb{E}[Y_i^2-W_i^2|\mathcal{H}^{<v_i}] \, \Big| \,\Big] \nonumber \\
		 & \le
		\EE \Big[\, \EE[|Y_i-W_i||Y_i+W_i|\Big| \mathcal{H}^{<v_i}]\,\Big] \nonumber               \\
		 & =
		\EE \big[\, |Y_i-W_i||Y_i+W_i| \, \big] \nonumber                                          \\
		\label{eq:Step5a}
		 & \le
		\EE[(Y_i-W_i)^2]^{1/2} \EE[(Y_i+W_i)^2]^{{1/2}} \,.
	\end{align}
	We will show that the second expectation in the final expression above
	is uniformly bounded in $i$, while the first converges to $0$ as the
	distance of $v_i$ from the coordinate axes grows. The intuition for
	the second point is that $Y_i$ is a conditional expectation of $W_i$
	given $\mathcal{H}_\Lambda$; in other words, to get $W_i$ from $Y_i$
	we are just taking an expectation over the randomness coming from
	$(X_w)_{w\in \mathbb{Z}^2\setminus \Lambda, w\le v_i}$, which makes
	little difference if $v_i$ is far from both axes. Since as
	$\Lambda\uparrow \Z^2$ the proportion of vertices in $\Lambda$ that
	are close to the coordinate axes goes to $0$, we obtain that the sum
	in \eqref{eq:step5} converges to $0$. {The formal details of these
			three steps follow.}

	For uniform boundedness of $\EE[(Y_i+W_i)^2]$, we use \eqref{eq:YW}
	and the fact that the conditional expectation is a contraction in
	$L^2$ to see that
	\begin{equation}\label{eq:sum2bound}
		\EE [(Y_i+W_i)^2]
		\leq 4 \EE [W_i^2]
		\leq 64\lambda^2 \EE [X_{v_i}^2] =: c^2 < \infty\,,
	\end{equation}
	where for the last inequality we have used that $(X_{v_i})_i$ are i.i.d.\ together with \eqref{eq:Wbound}.

	To control $\EE [(Y_i-W_i)^2]$, using \eqref{eq:YW} again, we have
	\[
		\EE[(Y_i-W_i)^2] =
		\EE \big[ \big( W_i-\EE[W_i|\mathcal{H}_{\Lambda}]\big)^2 \big] \,,
	\]
	which can be expressed in terms of $W_1$ or $W_2$, depending on the
	parity of $v_i$.  Indeed, by \eqref{eq:Wshift}, for all $i$ we can
	express $W_i$ in terms of a fixed site and a shift in the following
	way:
	\begin{equation*}
		W_i(X) =
		\begin{cases}
			W_1 (T_{v_i}X)       & \text{ if $v_i$ is even} \,, \\
			W_2 (T_{v_i-(1,0)}X) & \text{ if $v_i$ is odd} \,.
		\end{cases}
	\end{equation*}
	Therefore,
	defining $\mathcal{H}_{\Lambda - v} = \sigma (X_w : w+v \in \Lambda )$,
	we have
	\begin{equation}
		\EE[(Y_i-W_i)^2] =
		\begin{cases}
			\EE [ (W_1-\EE[W_1|\mathcal{H}_{\Lambda - v_i}])^2]       & \text{ if $v_i$ is even} \,, \\
			\EE [ (W_2-\EE[W_2|\mathcal{H}_{\Lambda - v_i+(1,0)}])^2] & \text{ if $v_i$ is odd} \,.
		\end{cases}
	\end{equation}
	As $W_1$ is measurable with respect to $\mathcal{H}^{\leq \oo} = \sigma( X_v \, : \, v \leq \oo )$,
	for $v=(v_x,v_y)$ even we have
	\[
		\EE [ (W_1-\EE[W_1|\mathcal{H}_{\Lambda - v}])^2]
		=
		\EE [ (W_1-\EE[W_1| \, \sigma(X_{w} : w \leq \oo \text{ and } w \in \Lambda-v )])^2] \,.
	\]
	The same holds for $v$ odd, replacing $W_1$ and $\oo$ by $W_2$ and
	$\oo'=(1,0)$, respectively. Now, define
	\[
		\mathcal{H}_v^{\leq \oo}
		=
		\sigma (X_w \, : \, w \leq \oo \text{ and } w \in [-v_x,0]\times[-v_y, 0]) \,,
	\]
	which increases to $\mathcal{H}^{\leq \oo}$ as $m(v):=\min(v_x,v_y) \to \infty$ and notice that $ \mathcal{H}_v^{\leq \oo} \subseteq \sigma(X_{w} : w \leq \oo \text{ and } w \in \Lambda-v ) $.
	Then, since $W_1$ is $\mathcal{H}^{\leq 0}$-measurable and in $L^2(\mathbb{P})$, for $v$ even, as $m(v) \to \infty$, we have
	\begin{equation*}
		\EE [ (W_1-\EE[W_1|\sigma(X_{w} : w \leq \oo \text{ and } w \in \Lambda-v )])^2] \leq
		\EE[ (W_1 -\EE [W_1 | \mathcal{H}_v^{\leq \oo}] )^2] \to 0 \,,
	\end{equation*}
	Applying the same reasoning with $W_2$ and the past up to $\mathbf{o'} = (1,0)$ yields the same conclusion for odd $v$.
	Thus, we deduce that
	\begin{equation}
		\label{eq:cvYWi}
		\mathbb{E}[(Y_i-W_i)^2] \to 0 \,, \quad \text{ as } m(v_i)\to \infty \,.
	\end{equation}

	To prove~\eqref{eq:step5}, fix $\delta>0$.
	We split the sum into boundary and interior parts by choosing $R$
	sufficiently large so that $\mathbb{E}[(Y_i-W_i)^2] \le \delta$ for
	all $v_i$ such that $m(v_i)>R$; this is possible by
	\eqref{eq:cvYWi}. Then, by \eqref{eq:Step5a} and~\eqref{eq:sum2bound}
	\begin{align*}
		\mathbb{E} & \left[\left|\frac{1}{|\Lambda|}\sum_{i=1}^{|\Lambda|}\left(\mathbb{E}[Y_i^2 \mid \mathcal{H}_\Lambda^{<v_i}]-\mathbb{E}[W_i^2\mid \mathcal{H}^{<v_i}]\right)\right|\right]
		\le \frac{c}{|\Lambda|}\sum_{i=1}^{|\Lambda|} \mathbb{E}[(Y_i-W_i)^2]^{1/2}                                                                                                                \\
		           & = \frac{c}{|\Lambda|}\sum_{i: v_i\in \Lambda, m(v_i)\le R} \mathbb{E}[(Y_i-W_i)^2]^{1/2}+\frac{c}{|\Lambda|}\sum_{i: v_i\in \Lambda, m(v_i)>R} \mathbb{E}[(Y_i-W_i)^2]^{1/2}.
	\end{align*}
	In the final expression, the second term is less than $c\delta$ by
	choice of $R$, and the first term is at most $c\delta$ for $\Lambda$
	big enough, since the proportion of vertices $v$ in $\Lambda$ with
	$m(v)<R$ goes to $0$ and we have $\EE[(Y_i-W_i)^2] \leq c^2$ for all
	$i$ by the same argument used to prove \eqref{eq:sum2bound}.
	As $\delta>0$ is arbitrary, we get the desired convergence (uniformly in $\Lambda$).

	\subsubsection*{Step 6} Combining Steps $3, 4$ and $5$ we see that for any $a>0$
	\[
		\frac{1}{|\Lambda|}\sum_{i=1}^{|\Lambda|}\mathbb{E}[Y_i^2 \mathds{1}_{t|Y_i|\le a\sqrt{|\Lambda|}}\mid \mathcal{H}_\Lambda^{<v_i}]\stackrel{\PP}{\longrightarrow}  b^2
	\]
	as $\Lambda \nearrow \mathbb{Z}^2$. Combining this with
	Step~2 (i.e., \eqref{eq:lapbound}), we see that
	\[
		\liminf_{\Lambda\nearrow \mathbb{Z}^2} \mathbb{E}\left[\exp\left({\frac{tF_\Lambda}{\sqrt{|\Lambda|}}}\right)\right]\ge \exp \Big( \frac{t^2b^2(1-f(a))}{2}\Big) \,.
	\]
	for any $a>0$, and taking $a\to 0$ completes the proof.
\end{proof}

\small
\bibliographystyle{abbrv}
\bibliography{biblio}

@misc{chen2025stabilitylongrangeorderdisordered,
      title={The stability of long-range order in disordered systems: A generalized {D}ing-{Z}huang argument}, 
      author={Yejia Chen and Jianwen Zhou and Ruifeng Liu and Hai-Jun Zhou},
      year={2025},
      eprint={2507.11445},
      archivePrefix={arXiv},
      primaryClass={math-ph},
      url={https://arxiv.org/abs/2507.11445}, 
}

@misc{chowdhury2025decouplingclustersindependentsets,
      title={Decoupling of clusters in independent sets in a percolated hypercube}, 
      author={Mriganka Basu Roy Chowdhury and Shirshendu Ganguly and Vilas Winstein},
      year={2025},
      eprint={2511.07350},
      archivePrefix={arXiv},
      primaryClass={math.PR},
      url={https://arxiv.org/abs/2511.07350}, 
}

@article{jenssen2024refinedgraphcontainerlemma,
  title={A refined graph container lemma and applications to the hard-core model on bipartite expanders},
  author={Jenssen, Matthew and Malekshahian, Alexandru and Park, Jinyoung},
  journal={arXiv:2411.03393},
  year={2024}
}

@article{geisler2025countingindependentsetspercolated,
  title={Counting independent sets in percolated graphs via the Ising model},
  author={Geisler, Anna and Kang, Mihyun and Sarantis, Michail and Wdowinski, Ronen},
  journal={arXiv:2504.08715},
  year={2025}
}

@article{ding2024long,
  title={Long range order for random field {I}sing and {P}otts models},
  author={Ding, Jian and Zhuang, Zijie},
  journal={Communications on Pure and Applied Mathematics},
  volume={77},
  number={1},
  pages={37--51},
  year={2024},
  publisher={Wiley Online Library}
}

@article{ding2024long2,
  title={Long range order for three-dimensional random field {I}sing model throughout the entire low temperature regime},
  author={Ding, Jian and Liu, Yu and Xia, Aoteng},
  journal={Inventiones mathematicae},
  volume={238},
  number={1},
  pages={247--281},
  year={2024},
  publisher={Springer}
}

@article{dario2024quantitative,
  title={Quantitative disorder effects in low-dimensional spin systems},
  author={Dario, Paul and Harel, Matan and Peled, Ron},
  journal={Communications in Mathematical Physics},
  volume={405},
  number={9},
  pages={212},
  year={2024},
  publisher={Springer}
}

@article{aizenman2020exponential,
  title={Exponential decay of correlations in the 2{d} random field {I}sing model},
  author={Aizenman, Michael and Harel, Matan and Peled, Ron},
  journal={Journal of Statistical Physics},
  volume={180},
  number={1},
  pages={304--331},
  year={2020},
  publisher={Springer}
}

@article{helmuth2023approximation,
  title={Approximation Algorithms for the Random Field {I}sing Model},
  author={Helmuth, Tyler and Lee, Holden and Perkins, Will and Ravichandran, Mohan and Wu, Qiang},
  journal={SIAM Journal on Discrete Mathematics},
  volume={37},
  number={3},
  pages={1610--1629},
  year={2023},
  publisher={SIAM}
}

@article{alaoui2023fast,
  title={Fast relaxation of the random field Ising dynamics},
  author={Alaoui, Ahmed El and Eldan, Ronen and Gheissari, Reza and Piana, Arianna},
  journal={arXiv:2311.06171},
  year={2023}
}

@article{ding2023correlation,
  title={Correlation length of the two-dimensional random field {I}sing model via greedy lattice animal},
  author={Ding, Jian and Wirth, Mateo},
  journal={Duke Mathematical Journal},
  volume={172},
  number={9},
  pages={1781--1811},
  year={2023},
  publisher={Duke University Press}
}

@article{goldberg2007complexity,
  title={The complexity of ferromagnetic {I}sing with local fields},
  author={Goldberg, Leslie Ann and Jerrum, Mark},
  journal={Combinatorics, Probability and Computing},
  volume={16},
  number={1},
  pages={43--61},
  year={2007},
  publisher={Cambridge University Press}
}

@article{ding2021exponential,
  title={Exponential decay of correlations in the two-dimensional random field {I}sing model},
  author={Ding, Jian and Xia, Jiaming},
  journal={Inventiones mathematicae},
  volume={224},
  number={3},
  pages={999--1045},
  year={2021},
  publisher={Springer}
}

@article{CannonHelmuthPerkins2024,
  author    = {Sarah Cannon and Tyler Helmuth and Will Perkins},
  title     = {Pirogov–{S}inai Theory Beyond Lattices},
  journal    = {arXiv:2411.07809},
  year      = {2024},
  note      = {\url{https://arxiv.org/abs/2411.07809}},
}

@article{dyer2004relative,
  title={The relative complexity of approximate counting problems},
  author={Dyer, Martin and Goldberg, Leslie Ann and Greenhill, Catherine and Jerrum, Mark},
  journal={Algorithmica},
  volume={38},
  number={3},
  pages={471--500},
  year={2004},
  publisher={Springer}
}

@article{KronenbergSpinka2022,
  author    = {Gal Kronenberg and Yinon Spinka},
  title     = {Independent sets in random subgraphs of the hypercube},
  journal   = {arXiv preprint},
  eprint    = {arXiv:2201.06127},
  year      = {2022},
  note      = {\url{https://arxiv.org/abs/2201.06127}},
}

@article{AizenmanWehr1990,
  author       = {Michael Aizenman and Jan Wehr},
  title        = {Rounding effects of quenched randomness on first‐order phase transitions},
  journal      = {Communications in Mathematical Physics},
  volume       = {130},
  number       = {3},
  pages        = {489--528},
  year         = {1990},
  doi          = {10.1007/BF02096933},
}

@article{dobrushin1968problem,
  title={The problem of uniqueness of a {G}ibbsian random field and the problem of phase transitions},
  author={Dobrushin, Roland L’vovich},
  journal={Functional Analysis and its Applications},
  volume={2},
  number={4},
  pages={302--312},
  year={1968},
  publisher={Springer}
}

@article{Taylor1985,
  author       = {David E. Taylor and Ellen D. Williams and Robert L. Park and N. C. Bartelt and T. L. Einstein},
  title        = {Two‐dimensional ordering of chlorine on {A}g(100)},
  journal      = {Physical Review B},
  volume       = {32},
  number       = {7},
  pages        = {4653--4662},
  year         = {1985},
  doi          = {10.1103/PhysRevB.32.4653},
}

@article{FKG1971,
  author    = {C. M. Fortuin and P. W. Kasteleyn and J. Ginibre},
  title     = {Correlation inequalities on some partially ordered sets},
  journal   = {Communications in Mathematical Physics},
  volume    = {22},
  number    = {2},
  pages     = {89--103},
  year      = {1971},
  doi       = {10.1007/BF01651325}
}

@article{Holley1974,
  author    = {Richard Holley},
  title     = {Remarks on the {FKG} inequalities},
  journal   = {Communications in Mathematical Physics},
  volume    = {36},
  number    = {3},
  pages     = {227--231},
  year      = {1974},
  doi       = {10.1007/BF01646490}
}

@book{Georgii1988,
  author    = {Hans-Otto Georgii},
  title     = {Gibbs Measures and Phase Transitions},
  series    = {De Gruyter Studies in Mathematics},
  volume    = {9},
  publisher = {Walter de Gruyter},
  address   = {Berlin},
  year      = {1988},
  doi       = {10.1515/9783110850772},
  isbn      = {978-3-11-011426-3}
}

@article{BergSteif1999,
  author    = {J. van den Berg and J. E. Steif},
  title     = {Percolation and the hard-core lattice gas model},
  journal   = {Stochastic Processes and their Applications},
  volume    = {49},
  number    = {2},
  pages     = {179--197},
  year      = {1994},
  url       = {https://ir.cwi.nl/pub/1496/1496D.pdf},
  doi       = {10.1016/0304-4149(94)90066-3}
}

@incollection{GeorgiiHaeggMaes2001,
  author    = {Hans-Otto Georgii and Olle H{\"a}ggstr{\"o}m and Christian Maes},
  title     = {The random geometry of equilibrium phases},
  booktitle = {Phase Transitions and Critical Phenomena},
  series    = {Phase Transit. Crit. Phenom.},
  volume    = {18},
  pages     = {1--142},
  publisher = {Academic Press},
  year      = {2001},
  isbn      = {0-12-220318-6},
  doi       = {10.1016/S1062-7901(01)80008-2},
  mrclass   = {82B20 (82-02 82B26 82B31 82B43)},
  mrnumber  = {2014387},
  mrreviewer= {Bruno Nachtergaele},
  url       = {https://doi.org/10.1016/S1062-7901(01)80008-2}
}

@book{Bovier06, 
	place={Cambridge}, 
	series={Cambridge Series in Statistical and Probabilistic Mathematics}, 
	title={Statistical Mechanics of Disordered Systems: A Mathematical Perspective}, 
	publisher={Cambridge University Press},
	author={Bovier, Anton},
	year={2006}, collection={Cambridge Series in Statistical and Probabilistic Mathematics}
}

@book{krengel2011ergodic,
url = {https://doi.org/10.1515/9783110844641},
title = {Ergodic Theorems},
author = {Ulrich Krengel},
publisher = {De Gruyter},
doi = {doi:10.1515/9783110844641},
isbn = {9783110844641},
year = {1985},
lastchecked = {2025-11-19}
}

\end{document}